\documentclass[12pt]{article}
\usepackage{amssymb,amsthm,amsfonts,latexsym,tikz,hyperref,start4,ytableau}
\usepackage[hmargin=1in,vmargin=1in]{geometry}
\usepackage{amstext,amscd,amsbsy}
\usepackage{enumerate,graphicx,tikz,float,color}
\usepackage{amsopn,listings}
\usepackage{algpseudocode}
\usepackage[justification=centering]{caption}
\usepackage[normalem]{ulem}
\usepackage{amsmath,pgfopts}

\newtheorem{thm}{Theorem}[section]
\newtheorem{prop}[thm]{Proposition}
\newtheorem{cor}[thm]{Corollary}
\newtheorem{lem}[thm]{Lemma}
\newtheorem{conj}[thm]{Conjecture}

\newtheorem{question}[thm]{Question}
\newtheorem{defn}[thm]{Definition}

\renewcommand{\P}{\mathbb{P}}

\begin{document}
\pagestyle{plain}

\title{On the $1/3$--$2/3$ Conjecture
}
\author{
Emily J. Olson \\[-5pt]
\small Department of Mathematics, Millikin University,\\[-5pt]
\small Decatur, IL 62522, USA, {\tt ejolson@millikin.edu}
\and
Bruce E. Sagan\\[-5pt]
\small Department of Mathematics, Michigan State University,\\[-5pt]
\small East Lansing, MI 48824-1027, USA, {\tt sagan@math.msu.edu}
}

\date{\today\\[10pt]
	\begin{flushleft}
	                                       \\[5pt]
	\end{flushleft}}

\maketitle

\begin{abstract}
Let $(P,\leq)$ be a finite poset (partially ordered set), where $P$ has cardinality $n$. Consider linear extensions of $P$ as permutations $x_1x_2\cdots x_n$ in one-line notation. For distinct elements $x,y\in P$, we define $\P(x\prec y)$ to be the proportion of linear extensions of $P$ in which $x$ comes before $y$. 
For $0\leq \alpha \leq \frac{1}{2}$, we say $(x,y)$ is an $\alpha$-balanced pair if $\alpha \leq \P(x\prec y) \leq 1-\alpha.$
The $1/3$--$2/3$ Conjecture states that every finite partially ordered set which is not a chain has a $1/3$-balanced pair.  We make progress on this conjecture by showing that it holds for certain families of posets.   These include  lattices such as the Boolean, set partition, and subspace lattices; partial orders that arise from a Young diagram; and some partial orders of dimension $2$.  We also consider various posets which satisfy the stronger condition of having a $1/2$-balanced pair.
For example, this happens when the poset has an automorphism with a cycle of length $2$. 
Various questions for future research are posed.

\end{abstract}


\section{Introduction}

Let $(P,\leq)$ be a poset, and let $n$ be the cardinality of $P$. A \emph{linear extension} is a total order $x_1\prec x_2\prec \cdots \prec x_n$ on the elements of $P$ such that $x_i\prec x_j$ if $x_i<_P x_j$; more compactly, we can view a linear extension as a permutation $x_1x_2\cdots x_n$ in one-line notation. For distinct elements $x,y\in P$, we define $\P(x\prec y)$ to be the proportion of linear extensions of $P$ in which $x$ comes before $y$. 
For $0\le \alpha \le \frac{1}{2}$, we say $(x,y)$ is an \emph{$\alpha$-balanced pair} if 
\[ \alpha \leq \P(x\prec y) \leq 1-\alpha,\]
and that $P$ is \emph{$\alpha$-balanced} if it has some $\alpha$-balanced pair. Notice that if $(x,y)$ is $\alpha$-balanced, then $(y,x)$ is $\alpha$-balanced as well. 

\begin{conj}[The $1/3$--$2/3$ Conjecture]\label{conj}
Every finite partially ordered set that is not a chain has a $1/3$-balanced pair.
\end{conj}

We can see, for instance, that the conjecture holds for the poset $P$ depicted in Figure~\ref{fig:6elemExample_width2}. This poset has 15 linear extensions, which are
\[ \begin{array}{rcccl}
	123456, & 123465, & 123645, & 124356, & 124365, \\
	124536, & 142356, & 142365, & 142536, & 213456, \\
	213465, & 213645, &	214356, & 214365, & 214536.	
	\end{array} \]	
The matrix on the right in the figure has as its $(i,j)$ entry the number of linear extensions of $P$ where $i$ comes before $j$.  The entries in bold give the pairs $(i,j)$ whose number of linear extensions with $i\prec j$ is between $1/3(15)=5$ and $2/3(15)=10$, thus satisfying the conjecture.

Conjecture~\ref{conj} was first proposed by Kislitsyn~\cite{Kis68} in 1968, although a number of resources attribute it to Fredman~\cite{Fre76}   and it was also independently discovered by Linial~\cite{Lin84}.
There are many types of posets for which the conjecture has already been proven. This includes posets of up to $11$ elements~\cite{Pec06}, posets with height $2$~\cite{TGF92}, semiorders~\cite{Bri89}, posets with each element incomparable to at most $6$ others~\cite{Pec08}, $N$-free posets~\cite{Zag12}, and posets whose Hasse diagram is a tree~\cite{Zag16}. 
If the conjecture is true, the bounds are  best possible, as seen by the poset $T$ in Figure~\ref{fig:3elem1relation}.
While the proof of the $1/3$ bound for a general poset remains elusive, in 1984 Kahn and Saks~\cite{KS84} proved that for any poset $P$, there is some pair $x,y\in P$ such that $\frac{3}{11}< \P(x\prec y)<\frac{8}{11}$. In 1995, Brightwell, Felsner, and Trotter~\cite{BFT95} improved the bound to be $\frac{5-\sqrt{5}}{10}\leq \P(x\prec y) \leq \frac{5+\sqrt{5}}{10}$.
 In~\cite{BFT95}, Conjecture~\ref{conj} is described as ``one of the most intriguing problems in combinatorial theory". 
 The interested reader can refer to Brightwell's 1999 survey~\cite{Bri99} for more information.

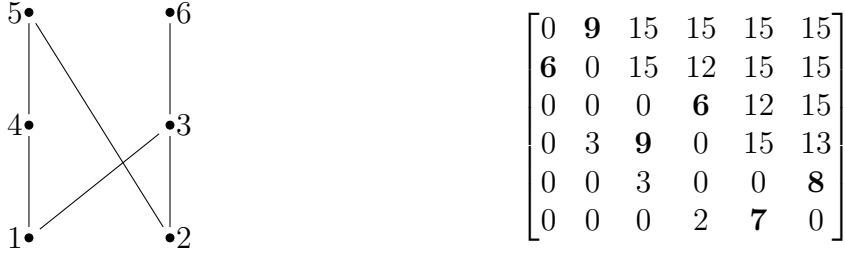
\begin{figure}[t!]
\centering
	\begin{minipage}{.45\textwidth}
		\centering
		\begin{tikzpicture}[scale=0.75]
		\filldraw (0,0) circle (2pt) node(v1){};
		\filldraw (0,2) circle (2pt) node(v4){};
		\filldraw (0,4) circle (2pt) node(v5){};
		\filldraw (2.5,0) circle (2pt) node(v2){};
		\filldraw (2.5,2) circle (2pt) node(v3){};
		\filldraw (2.5,4) circle (2pt) node(v6){};
		\draw (-0.25,0) node {$1$};
		\draw (-0.25,2) node {$4$};
		\draw (-0.25,4) node {$5$};
		\draw (2.75,0) node {$2$};
		\draw (2.75,2) node {$3$};
		\draw (2.75,4) node {$6$};
		\draw[-] (v1)--(v4)--(v5)--(v2)--(v3)--(v6);
		\draw[-] (v1)--(v3);
		\end{tikzpicture}	
	\end{minipage}
	\hspace{5mm}
	\begin{minipage}{.4\textwidth}
    \[ \begin{bmatrix} 0&\textbf{9}&15&15&15&15 \\
	\textbf{6}&0&15&12&15&15 \\
	0&0&0&\textbf{6}&12&15 \\
	0&3&\textbf{9}&0&15&13 \\
	0&0&3&0&0&\textbf{8} \\
	0&0&0&2&\textbf{7}&0
    \end{bmatrix} \]
	\end{minipage}
	\caption{A poset $P$ with 6 elements and a matrix counting its linear extensions}
	\label{fig:6elemExample_width2}	
\end{figure}

An alternative way of talking about this conjecture is as follows.
 We define the \emph{balance constant} of $P$  to be
\[ \delta(P) = \max_{x,y\in P} \min\{\P(x\prec y), \P(y\prec x)\} \]
For any poset $P$ not a chain, it must be that $0< \delta(P)\leq 1/2$. In the example in Figure~\ref{fig:6elemExample_width2}, $P$ has a balance constant of $\frac{7}{15} \approx 0.4667$. 
So $P$ has a $1/3$-balanced pair if and only if $P$ has a balance constant $\delta(P)\geq 1/3$. We will use these two phrases interchangeably, as does the literature.

The following  notation will give us yet another way of discussing the conjecture. Let $E(P)$ be the set of linear extensions of $P$ and $e(P)$ be the cardinality of $E(P)$. If $(P,\leq)$ is a poset and $x,y\in P$, let $P+xy$ denote the poset $(P,\leq')$, where $\leq'$ is the transitive closure of $\leq$ extended by the relation $x<y$.  So $\P(x\prec y)=e(P+xy)/e(P)$.  Note also that
\begin{equation}
\label{e(P)}
e(P+xy)+e(P+yx)=e(P).
\end{equation}

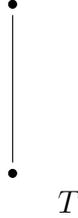
\begin{figure}[t!]
 \centering
 \begin{tikzpicture}[scale=0.75]
	\filldraw (0,0) circle (2pt) node(v1){};
	\filldraw (0,3) circle (2pt) node(v2){};
	\filldraw (2,1.5) circle (2pt) node(v3){};
	\draw (1,-0.5) node {$T$};
	\draw[-] (v1)--(v2);
	\end{tikzpicture}	
	\caption{The poset $T$ with three elements and one relation}
	\label{fig:3elem1relation}
\end{figure}

There are ideas of Zaguia which we will find useful in a number of our proofs. The following definitions were introduced in~\cite{Zag16}, although here we refer to them with different names which we find more descriptive.  Given $x$ in a poset $P$ we let $L_x$ and $U_x$ denote the {\em strict lower} and {\em strict upper order ideals} generated by $x$, that is,
$$
L_x=\{y\in P  \mid y< x\}
$$
and
$$
U_x=\{y\in P  \mid y> x\}.
$$

\begin{defn} \label{def:almosttwinpair}
Let $P$ be a poset and $x$ and $y$ be distinct elements of $P$.
\begin{enumerate}[(a)]
\item We call the pair $(x,y)$ \emph{twin elements} if $L_x = L_y$ and $U_x=U_y$. 
\item We call the pair $(x,y)$ \emph{almost twin elements} if the following two conditions hold in $P$ or in the dual of $P$:
\begin{enumerate}[(i)]
\item $L_x=L_y$, and
\item $U_x\ \backslash\ U_y$ and $U_y\ \backslash\ U_x$ are chains (possibly empty).
\end{enumerate}
\end{enumerate}
\end{defn}

The result of Zaguia's which we will need is as follows.

\begin{thm}[\cite{Zag16}]\label{thm:asymmetricpair}
A finite poset that has an almost twin pair of elements is $1/3$-balanced.\hfill\qed
\end{thm}

In fact,  Zaguia proved a more general result by relaxing the definition of an almost twin pair. But we will not need that level of generality here.
 It is important to keep in mind that many of the known results for the $1/3$--$2/3$ Conjecture are existence proofs and do not compute $\P(x\prec y)$ exactly for any pair $(x,y)$ in the given poset. This is particularly true in the case of  almost twin pairs of elements. 
Further, an almost twin pair need not be $1/3$-balanced, even though its existence implies that a $1/3$-balanced pair exists.

The rest of this paper is structured as follows.  In the next section, we will show that a poset with an automorphism containing a $2$-cycle is $1/2$-balanced.  In particular, every poset with twin elements is $1/2$-balanced.  The automorphism result is applied in Section~\ref{sec:Lattices} to various types of lattices including the Boolean lattice, set partition lattice, subspace lattice, and certain distributive lattices.  We also consider the lattice obtained by taking the product of two chains.  This last example is just the poset of a rectangular Young diagram and in Section~\ref{sec:OtherDiagrams} we show that the poset of any Young diagram, including those which are skew or shifted, is $1/3$-balanced.  Section~\ref{sec:Dim2} is devoted to showing that certain posets of dimension $2$ which satisfy a pattern avoidance condition have balance constant $1/2$.  We end with a section discussing posets which have balance constants near, but not equal to, $1/3$.  A number of questions concerning future research are scattered throughout.


\section{Automorphisms of Posets}\label{sec:Autos}

We first provide a proof of a simple observation about the linear extensions of a poset with an automorphism.

\begin{prop}~\label{prop:LEbijection}
An automorphism $\phi$ of a poset $P$ induces a bijection on $E(P)$. Further, $\P(x\prec y) = \P(\phi(x)\prec \phi(y))$ for all $x,y\in P$.
\end{prop}

\begin{proof}
Let $\phi: P\to P$ be any automorphism. This means that for $x,y \in P$, $x\le_P y$ if and only if $\phi(x)\le_P \phi(y)$. Now, let $\pi= a_1a_2\cdots a_n$ be a linear extension of $P$, and by the definition of linear extension, we know that if $a_i \le_P a_j,$ then $i\le j$. As $\phi$ is an automorphism, then we also have that if $\phi(a_i)\le_P \phi(a_j)$, then $i\le j$. This gives us, by definition, that $\phi(\pi)= \phi(a_1)\phi(a_2) \cdots \phi(a_n)$ is a linear extension of $P$. And the fact that $\phi$ is bijective implies that the induced map on $E(P)$ is as well.

We can also observe that the linear extensions with $x$ before $y$ map bijectively via $\phi$ to the linear extensions with $\phi(x)$ before $\phi(y)$. Hence, $\P(x\prec y) = \P(\phi(x)\prec \phi(y))$, as desired.
\end{proof}

In ~\cite{GHP}, Ganter, Hafner, and Poguntke prove that posets with a nontrivial automorphism satisfy Conjecture~\ref{conj}.

\begin{thm}[\cite{GHP}]~\label{thm:automorph}
If a poset $P$ has a non-trivial automorphism, then $P$ is $1/3$-balanced.\hfill\qed
\end{thm}

We will give a more refined version of this result by giving a condition on the automorphism which will ensure a balance constant of $1/2$.

\begin{prop}~\label{prop:auto2cycle}
If a poset $P$ has an automorphism with a cycle of length $2$, then $P$ is $1/2$-balanced. Further, if $x$ and $y$ are the elements in the cycle of length $2$, then $(x,y)$ is a $1/2$-balanced pair. 
\end{prop}

\begin{proof}
 Let $\phi:P\to P$ be an automorphism and $x,y\in P$ be distinct elements such that $\phi(x) = y$ and $\phi(y)=x$. Thus, using  Proposition~\ref{prop:LEbijection},  we see that
\[ e(P+xy) = e(P + \phi(x)\phi(y)) = e(P + yx). \]
Combining this with  equation~(\ref{e(P)}), we have
$$
e(P) = e(P+xy)+e(P+yx)\\
    = 2e(P+xy),
$$
and so $e(P+xy) = e(P)/2$. Hence, $(x,y)$ is a $1/2$-balanced pair, as desired.
\end{proof}

An example of a poset with an automorphism having a cycle of length $2$ is given in Figure~\ref{fig:auto_antiauto_examples}. Poset $P$ has a balance constant of $1/2$. A counterexample to the converse of Proposition~\ref{prop:auto2cycle} is also provided in Figure~\ref{fig:auto_antiauto_examples}. Poset $Q$ has a balance constant of $1/2$, as it has 12 linear extensions and $e(P+34)=6.$ However, we can see by inspection it has no nontrivial automorphisms. 

\begin{figure}[t!]
\centering
	\begin{minipage}{.35\textwidth}
		\centering
		\begin{tikzpicture}[scale=0.75]
		\filldraw (0,0) circle (2pt) node(v2){};
		\filldraw (2,0) circle (2pt) node(v1){};
		\filldraw (1,2) circle (2pt) node(v3){};
		\filldraw (0,4) circle (2pt) node(v5){};
		\filldraw (2,4) circle (2pt) node(v4){};
		\draw (-0.25,0) node {$2$};
		\draw (2.25,0) node {$1$};
		\draw (1.25,2) node {$3$};
		\draw (-0.25,4) node {$5$};
		\draw (2.25,4) node {$4$};
		\draw (1,-0.75) node {$P$};
		\draw[-] (v1)--(v3)--(v5);
		\draw[-] (v2)--(v3)--(v4);
		\end{tikzpicture}  
	\end{minipage}
	\hspace{5mm}
	\begin{minipage}{.35\textwidth}
		\centering
		\begin{tikzpicture}[scale=0.75]
		\filldraw (0,0) circle (2pt) node(v2){};
		\filldraw (0,2) circle (2pt) node(v4){};
		\filldraw (0,4) circle (2pt) node(v6){};
		\filldraw (2.5,0) circle (2pt) node(v1){};
		\filldraw (2.5,2) circle (2pt) node(v3){};
		\filldraw (2.5,4) circle (2pt) node(v5){};
		\draw (-0.25,0) node {$2$};
		\draw (-0.25,2) node {$4$};
		\draw (-0.25,4) node {$6$};
		\draw (2.75,0) node {$1$};
		\draw (2.75,2) node {$3$};
		\draw (2.75,4) node {$5$};
		\draw (1.25,-0.75) node {$Q$};
		\draw[-] (v1)--(v3)--(v6)--(v4)--(v2)--(v3)--(v5);
		\end{tikzpicture}
		%
	\end{minipage}
	\caption{The poset $P$ has an automorphism with cycle length $2$ and balance \\ constant $1/2$, while $Q$ has no nontrivial automorphisms and balance constant $1/2$}
	\label{fig:auto_antiauto_examples}	
\end{figure}
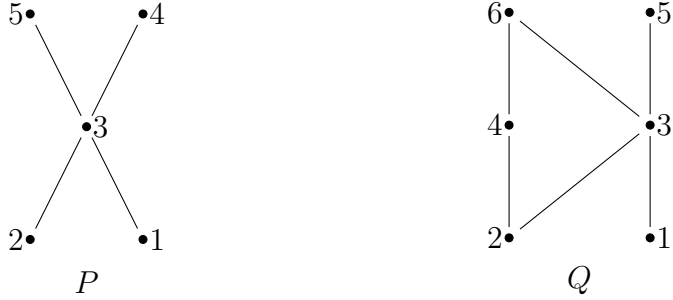

The following is a corollary to Proposition~\ref{prop:auto2cycle}.

\begin{cor}\label{cor:twinpair}
A poset $P$ with a twin pair of elements is $1/2$-balanced.
\end{cor}

\begin{proof}
Let $P$ be a poset with $x$ and $y$ a twin pair of elements. We can see that $P$ has a non-trivial automorphism  which fixes all elements except for $x$ and $y$ and interchanges $x$ and $y$. So, this poset has an automorphism with a cycle of length $2$ and we are done by Proposition~\ref{prop:auto2cycle}.
\end{proof}

While the above results depend on an automorphism of a poset, it is natural to ask if we can obtain results from other types of maps. Next, we consider anti-automorphisms $\sigma$.   So if $\sigma^{2}$ is not the identity, then $P$ has a non-trivial automorphism and we have the following immediate corollary to Theorem~\ref{thm:automorph}.
\begin{cor}
If $\sigma$ is an anti-automorphism on $P$ and $\sigma^{2}$ is non-trivial, then $P$ is $1/3$-balanced.
\hfill \qed
\end{cor}

We can also ask when an anti-automorphism guarantees a poset to be $1/2$-balanced. One such case is as follows.

\begin{prop}\label{prop:antiaut_1/2}
Let $\sigma:P\to P$ be an anti-automorphism. If $\sigma$ has $2$ fixed points, then $P$ is $1/2$-balanced.
\end{prop}

\begin{proof}
Let $P$ be a poset and $\sigma:P\to P$ be an anti-automorphism. Consider the bijection $\tau$ on linear extensions which takes a linear extension of $P$, applies $\sigma$ to each of its elements, and then reads the resulting sequence backwards. Observe that for any $x,y\in P$, the linear extensions with $x$ before $y$  map bijectively via $\sigma$ to the linear extensions with $\sigma(y)$ before $\sigma(x)$. Hence, $\P(x\prec y) = \P(\sigma(y)\prec \sigma(x))$. The proof is now completed in exactly the same way as the demonstration of Proposition~\ref{prop:auto2cycle}.
\end{proof}

We cannot weaken the assumption in Proposition~\ref{prop:antiaut_1/2}, since a unique fixed point in an anti-automorphism is not enough to guarantee that the poset is $1/2$-balanced. For an example, consider the poset $P$ and anti-automorphism $\sigma$ in Figure~\ref{fig:Antiaut_1fixed} where for $x\in P$ we place $\sigma(x)$ on the right in the same position as $x$ on the left. Any anti-automorphism of $P$, including $\sigma$ shown here, will have exactly $1$ fixed point, and computer calculations give us that $\delta(P)= \frac{711}{1431}\neq \frac{1}{2}$.

\begin{figure}[t!]
\centering
	\begin{tikzpicture}[scale=0.75]
		\filldraw (0,0) circle (2pt) node(v3){};
		\filldraw (0,2) circle (2pt) node(v6){};
		\filldraw (0,4) circle (2pt) node(v9){};
		\filldraw (2,0) circle (2pt) node(v2){};
		\filldraw (2,2) circle (2pt) node(v5){};
		\filldraw (2,4) circle (2pt) node(v8){};
		\filldraw (4,0) circle (2pt) node(v1){};
		\filldraw (4,2) circle (2pt) node(v4){};
		\filldraw (4,4) circle (2pt) node(v7){};
		\draw (-0.25,0) node {$3$};
		\draw (-0.25,2) node {$6$};
		\draw (-0.25,4) node {$9$};
		\draw (2.25,0) node {$2$};
		\draw (2.25,2) node {$5$};
		\draw (2.25,4) node {$8$};
		\draw (4.25,0) node {$1$};
		\draw (4.25,2) node {$4$};
		\draw (4.25,4) node {$7$};
		\draw (v9)--(v6)--(v3)--(v8)--(v5)--(v2)--(v7)--(v4)--(v1);
		\draw (0,4) .. controls (2.5,2.5) .. (4,0);
		\draw[->] (5,2)--(6,2);
		\draw (5.5,2.25) node {$\sigma$};
		\filldraw (7,0) circle (2pt) node(v3){};
		\filldraw (7,2) circle (2pt) node(v6){};
		\filldraw (7,4) circle (2pt) node(v9){};
		\filldraw (9,0) circle (2pt) node(v2){};
		\filldraw (9,2) circle (2pt) node(v5){};
		\filldraw (9,4) circle (2pt) node(v8){};
		\filldraw (11,0) circle (2pt) node(v1){};
		\filldraw (11,2) circle (2pt) node(v4){};
		\filldraw (11,4) circle (2pt) node(v7){};
		\draw (6.75,0) node {$8$};
		\draw (6.75,2) node {$5$};
		\draw (6.75,4) node {$2$};
		\draw (9.25,0) node {$9$};
		\draw (9.25,2) node {$6$};
		\draw (9.25,4) node {$3$};
		\draw (11.25,0) node {$7$};
		\draw (11.25,2) node {$4$};
		\draw (11.25,4) node {$1$};
		\draw (v9)--(v6)--(v3)--(v8)--(v5)--(v2)--(v7)--(v4)--(v1);
		\draw (7,4) .. controls (9.5,2.5) .. (11,0);
		\end{tikzpicture}
		\caption{$P$ and anti-automorphism $\sigma$ with $1$ fixed point}
		\label{fig:Antiaut_1fixed}		
\end{figure}
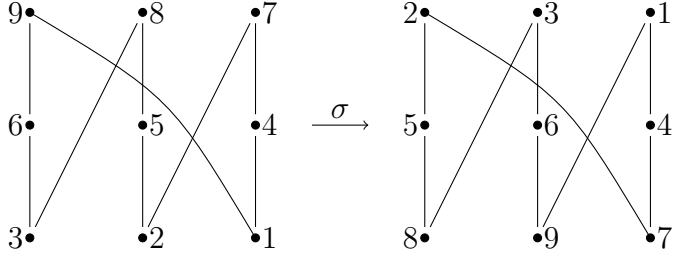

We can also see that the converse of Proposition~\ref{prop:antiaut_1/2} is not true, as evidenced by the counterexample in Figure~\ref{fig:auto_antiauto_examples}. Any anti-automorphism of the poset $P$ will have exactly 1 fixed point, and yet it is $1/2$-balanced.


\section{Lattices}\label{sec:Lattices}


\subsection{Boolean Lattices}

The Boolean lattice, $B_n$, consists of all subsets of $[n]:=\{1,2,\dots,n\}$ ordered by inclusion.
The poset $B_1$ is a chain, and so we only need to consider $n\geq 2$. We present the following as a corollary to Proposition~\ref{prop:auto2cycle}.

\begin{cor}
\label{cor:Bn}
For all $n\geq 2$, the Boolean lattice $B_n$ has an automorphism with a cycle of length $2$.  So the Boolean lattice is $1/2$-balanced.
\end{cor}

\begin{proof}
We will first describe an automorphism of $B_n$ using the symmetric difference operation $\Delta$. For $S\subseteq [n]$, consider $\phi:B_n\to B_n$ defined by
\[ \phi(S)=  \begin{cases}
    S\Delta \{1,2\},& \text{if } S\cap\{1,2\}=\{1\} \text{ or } S\cap\{1,2\} = \{2\} \\
    S, & \text{otherwise.}
\end{cases} \]
One can easily check that $\phi$ is an automorphism. And if $A=\{1\}$ and $B=\{2\}$, then $\phi(A) = B$ and $\phi(B)= A$. Hence, by Proposition~\ref{prop:auto2cycle}, $B_n$ has a $1/2$-balanced pair.
\end{proof}

\subsection{Set Partition Lattices}

The lattice, $\Pi_n$, consists of all partitions of $[n]$ ordered by refinement.  In writing set partitions, we separate subsets with slashes and dispense with set braces and commas.
For $n=1,2$ we have that $\Pi_n$ is a chain, and so will only consider $n\geq 3$.

\begin{cor}\label{lem:partitionlattice}
For $n\geq 3$, the set partition lattice $\Pi_n$ has an automorphism with a cycle of length $2$.  So the set partition lattice is $1/2$-balanced.
\end{cor}

\begin{proof}
We consider the map that sends a partition $\pi$ to the partition $\pi'$, where $\pi'$ has the same blocks as $\pi$ with the elements $1$ and $2$ interchanged. This is an automorphism of the lattice. Indeed, it is a bijection because it is an involution and swapping $1$ and $2$ preserves ordering by refinement.  
To see that this automorphism has a $2$-cycle, notice that the lattice contains partitions $\pi_1= 13/2/4/\cdots/n$ and $\pi_2 =1/23/4/\cdots/n$ since  $n\ge3$. 
Under the automorphism described above, $\pi_1$ and $\pi_2$ form a $2$-cycle. Hence, by Proposition~\ref{prop:auto2cycle}, the set partition lattice on $n$ elements is $1/2$-balanced when $n\ge3$.
\end{proof}

\subsection{Subspace Lattices}

Consider the $n$-dimensional vector space $\mathbb F_q^n$ over the Galois field  with $q$ elements. Let $L_n(q)$ denote the lattice of subspaces of $\mathbb F_q^n$ ordered by inclusion.  If $n\le 1$ then $L_n(q)$ is a chain.

\begin{cor}\label{cor:subspacelattic}
For $n\geq 2$, the subspace lattice $L_n(q)$ has an automorphism with a cycle of length $2$.  So the subspace lattice is $1/2$-balanced.
\end{cor}

\begin{proof}
Let $B = \{e_1, \ldots, e_n\}$ be the standard basis of $L_n(q)$.  Consider the linear transformation on $\mathbb F_q^n$ defined by the $n\times n$ matrix $M$ that is all zero except for ones in the $(1,2)$, $(2,1)$, and $(i,i)$ positions, $3\leq i\leq n$. Clearly, multiplying by $M$ sends $e_1$ to $e_2$, $e_2$ to $e_1$, and fixes all other basis elements of $\mathbb F_q^n$. If $U\in L_n(q)$ then let
$$
\phi(U)=MU=\{Mu \mid u\in U\}.
$$
It is now easy to check that $\phi$ is a well-defined automorphism of $L_n(q)$ which exchanges  the subspaces spanned by  $e_1$ and by $e_2$.  So we are done by Proposition~\ref{prop:auto2cycle}.
\end{proof}

\subsection{Distributive Lattices}\label{sec:Distributive}

By the Fundamental Theorem on Distributive Lattices, every distributive lattice is isomorphic to the lattice of lower order ideals of some poset $P$ ordered by inclusion.  So it would be interesting to determine results about $J(P)$, the distributive lattice corresponding to a poset $P$, depending on properties of $P$. Unfortunately, it is not true that if $P$ is $1/2$-balanced, then $J(P)$ is $1/2$-balanced as well. An example can be seen in Figure~\ref{fig:J234Chain}. While $P$ is $1/2$-balanced by the pair $(1,3)$, $J(P)$ is not $1/2$-balanced, as evidenced in the chart in Figure~\ref{fig:J234Chain} whose entries are $e(P+xy)$ for every $x$ and $y$ not comparable in $J(P)$. Since $J(P)$ has $14$ linear extensions, we can see no pair is $1/2$-balanced.
However, adding an extra condition allows us to prove that $J(P)$ is $1/2$-balanced.

\begin{figure}[t!]
\centering
\begin{minipage}{.55\textwidth}
\centering
    \begin{tikzpicture}[scale=0.75]
  	\filldraw (-7,2) circle (2pt) node(w2){};
  	\filldraw (-7,3) circle (2pt) node(w3){};
  	\filldraw (-7,4) circle (2pt) node(w4){};
  	\filldraw (-6,3) circle (2pt) node(w1){};
  	\draw (-7.25,2) node {$2$};
  	\draw (-7.25,3) node {$3$};
  	\draw (-7.25,4) node {$4$};
  	\draw (-5.75,3) node {$1$};
  	\draw (-6,1) node {$P$};
  	\draw[-] (w2) -- (w3) -- (w4);
  	\filldraw (0,0) circle (2pt) node(v0){};
	\filldraw (2,1) circle (2pt) node(v1){};
	\filldraw (-2,1) circle (2pt) node(v2){};
	\filldraw (2,3) circle (2pt) node(v3){};
	\filldraw (-2,3) circle (2pt) node(v4){};
	\filldraw (2,5) circle (2pt) node(v5){};
	\filldraw (-2,5) circle (2pt) node(v6){};
	\filldraw (0,6) circle (2pt) node(v7){};
	\draw (0,-0.4) node {$\emptyset$};
	\draw (2.5,1) node {$\{1\}$};
	\draw (-2.5,1) node {$\{2\}$};
	\draw (2.75,3) node {$\{1,2\}$};
	\draw (-2.75,3) node {$\{2,3\}$};
	\draw (3,5) node {$\{1,2,3\}$};
	\draw (-3,5) node {$\{2,3,4\}$};
	\draw (0,6.35) node {$\{1,2,3,4\}$};
	\draw (-1.5,-0.5) node {$J(P)$};
	\draw[-] (v0) -- (v1) -- (v3) -- (v5) -- (v7);
	\draw[-] (v0) -- (v2) -- (v3);
	\draw[-] (v2) -- (v4) -- (v5);
	\draw[-] (v4) -- (v6) -- (v7);
	\end{tikzpicture}
\end{minipage}
\hspace{3mm}
\begin{minipage}{.4\textwidth}
	\centering
	\[ \begin{array}{r|ccc}
	_x\diagdown^y & \{2\} & \{2,3\} & \{2,3,4\} \\ \hline
	\{1\} & 5 & 10 & 13 \\
	\{1,2\} & & 4 & 10 \\
	\{1,2,3\} & & & 5  \\
	\end{array}\]
\end{minipage}
	\caption{A $4$ element poset $P$, its corresponding $J(P)$, and a chart with values $e(P+xy)$}
	\label{fig:J234Chain}		
\end{figure}
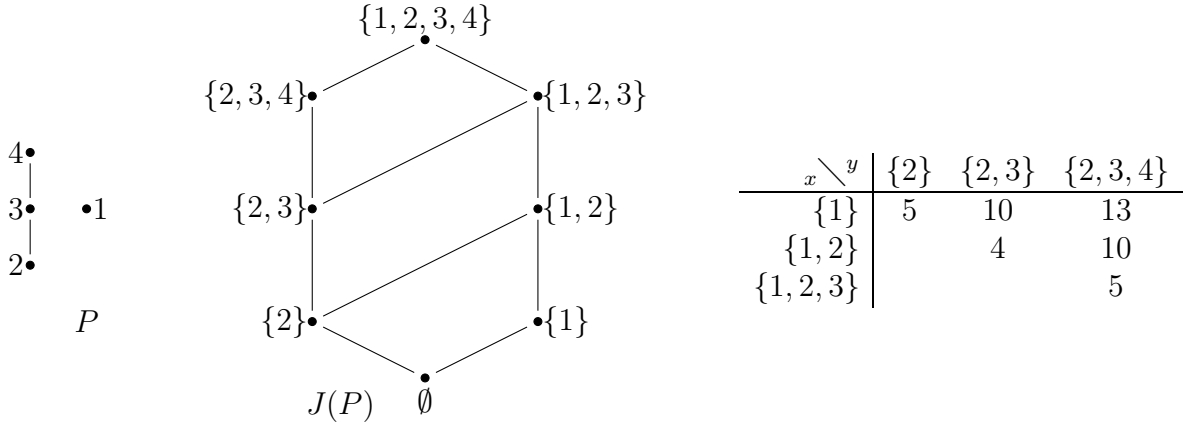

\begin{prop}\label{prop:distributive}
If $P$ has an automorphism with cycle of length $2$, then $J(P)$ is $1/2$-balanced.
\end{prop}

\begin{proof}
	Let $\phi:P\to P$ be an automorphism with $\phi(x)=y$ and $\phi(y)=x$ for some distinct $x,y\in P$. This induces an automorphism $\overline{\phi}$ of $J(P)$, given by $\overline{\phi}(I) = \{\phi(w):w\in I\}$ for $I\in J(P)$. 
	We claim that $\overline{\phi}$ has a cycle of length $2$, namely that $\overline{\phi}(I_x) = I_y$ and $\overline{\phi}(I_y) =I_x$ where $I_x,I_y$ are the  lower order ideals generated by $x,y$ respectively.
	
	We will show $\overline{\phi}(I_x) = I_y$, as the proof of the other equality is similar. Let $z\in \overline{\phi}(I_x)$, so $z = \phi(w)$ for some $w\in I_x$. This means that $w\le  x$, and so $\phi(w)\le \phi(x) = y$. Therefore, $z\le y$ and we have $z\in I_y$. Hence, $\overline{\phi}(I_x)\subseteq I_y$. The proof of the other set containment is similar.
	Thus, we have proven the claim. Since $\overline{\phi}$ has a cycle of length $2$, by Proposition 2.3, $J(P)$ is 1/2-balanced.
%
\end{proof}

This leads to another proof that Boolean lattices are $1/2$-balanced.

\begin{cor}
For $n\geq 2$, the Boolean lattice $B_n$ is $1/2$-balanced.
\end{cor}

\begin{proof}
The Boolean lattice $B_n$ is the distributive lattice corresponding to the poset $P$ with $n$ elements and no relations. There is an automorphism on $P$ that swaps elements $1$ and $2$ and is the identity on the remaining elements. Since $P$ has an automorphism with a cycle of length $2$, $B_n$ is $1/2$-balanced  by Proposition~\ref{prop:distributive}.
\end{proof}

We note that if we use the construction in Proposition~\ref{prop:distributive} on the $\phi$ from the proof  of the previous corollary, then the resulting $\overline{\phi}$ is exactly the map used to prove Corollary~\ref{cor:Bn}.  Also, the ideas in this subsection raise some interesting questions.

\begin{question}
Are all distributive lattices  $1/3$-balanced? What other characteristics of $P$  would imply that $J(P)$ is $1/2$-balanced?
\end{question}

\subsection{Products of Two Chains}\label{sec:2Chains}

Let $C_n$ be the chain with $n$ elements. This section will be concerned with the product of two chains $C_m$ and $C_n$, with $m,n\geq 2$.  See Figure~\ref{fig:C3xC4+3x4diagram} for the Hasse diagram of $C_3\times C_4$.  Such products can also be interpreted in terms of Young diagrams.  An {\em integer partition} is a weakly decreasing sequence $\la=(\la_1,\dots,\la_l)$ of positive integers.  The corresponding {\em shape} is an array of $l$ rows of left-justified boxes, also called cells, with $\la_i$ boxes in row $i$.
The Young diagram of shape $\lambda = (4,4,4)$ is displayed on the right in Figure~\ref{fig:C3xC4+3x4diagram} and the diagram of $\la=(4,4,2)$ is given  in Figure~\ref{fig:HooklengthEx} (ignoring the entries in the boxes for now).  
We often make no distinction between an integer partition and its Young diagram.
Let $(i,j)$ denote the box in row $i$ and column $j$ of $\la$.  Then we turn this Young diagram into a poset by ordering the boxes component-wise: $(i,j)\le (i',j')$ if and only if $i\le i'$ and $j\le j'$.  It should now be clear that the posets $C_m\times C_n$ and  $\la=(n^m)$ are isomorphic where $n^m$ represents $n$ repeated $m$ times as in Figure~\ref{fig:C3xC4+3x4diagram}.

The linear extensions of $\la$ can be thought of as a certain type of tableau.  A {\em standard Young tableau (SYT) of shape $\la$}, $Y$,  is a filling of the boxes of $\la$ with the integers $1,\dots,n=\sum_i\la_i$ so that the rows and columns increase.
An SYT of shape $\la=(4,4,2)$ is displayed on the left in Figure~\ref{fig:HooklengthEx}.  The SYT of shape $\la$ are in bijection with the linear extensions of $\la$ where $k$ is the entry in box $(i,j)$ if and only if $(i,j)$ is the $k$th element of the linear extension.  We will freely go back and forth between these two viewpoints.

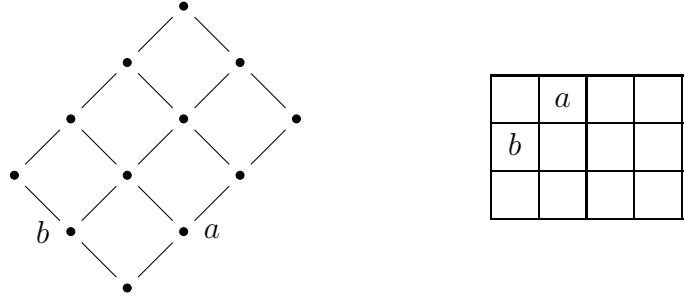
\begin{figure}[t!]
\centering
	\begin{minipage}{.3\textwidth}
 		\centering
  		\begin{tikzpicture}[scale=0.75]
  		\filldraw (3,0) circle (2pt) node(v0){};
  		\filldraw (4,1) circle (2pt) node(v1){};
  		\draw (1.5,1) node {$b$};
  		\draw (4.5,1) node {$a$};
  		\filldraw (2,1) circle (2pt) node(v2){};
  		\filldraw (5,2) circle (2pt) node(v3){};
  		\filldraw (3,2) circle (2pt) node(v4){};
  		\filldraw (1,2) circle (2pt) node(v5){};
  		\filldraw (4,3) circle (2pt) node(v6){};
  		\filldraw (2,3) circle (2pt) node(v7){};
  		\filldraw (6,3) circle (2pt) node(v8){};
  		\filldraw (3,4) circle (2pt) node(v9){};
  		\filldraw (5,4) circle (2pt) node(v10){};
  		\filldraw (4,5) circle (2pt) node(vT){};
  		\draw[-] (v0)--(v1)--(v3)--(v8)--(v10)--(vT)--(v9)--(v7)--(v5)--(v2)--(v0);
  		\draw[-] (v2)--(v4)--(v6)--(v10);
  		\draw[-] (v1)--(v4)--(v7);
  		\draw[-] (v3)--(v6)--(v9);
		\end{tikzpicture}
	\end{minipage}
\hspace{5mm}
	\begin{minipage}{.3\textwidth}
  		\centering
  		\ytableaushort{\none a, b}*{4,4,4}
	\end{minipage}
	\caption{The poset $C_3\times C_4$ and its corresponding diagram}
		\label{fig:C3xC4+3x4diagram}
\end{figure}

Unlike many other demonstrations that a poset is $1/3$-balanced, our proof for $C_m\times C_n$ finds the exact value of $\P(a\prec b)$ for a pair of elements $(a,b)$. It also provides a nice application of the famous Hooklength Formula, equation~(\ref{hook}) below.
Consider the cells $a=(1,2)$ and $b=(2,1)$ as labeled in Figure~\ref{fig:C3xC4+3x4diagram}. 
In order to compute how many linear extensions of $C_m\times C_n$ have $a\prec b$, we will compute how many SYT have cell $(1,2)$ filled with a smaller number than cell $(2,1)$.   
Since the entry $2$ must go in one of these two cells, this assumption forces  the SYT to have the $(1,1)$ cell filled with a $1$ and the $(1,2)$ cell filled with a $2$. Flipping and rotating by $180$ degrees, one sees that this is equivalent to counting the SYT of shape $(n^{m-1}, n-2)$.

To prove the next lemma, we will need the hooklength formula for $f^\lambda$, the number of SYT of shape $\lambda$. For a given cell $(i,j)$ in a diagram of shape $\lambda$, its {\em hook} is the set of all the cells weakly to its right and in the same row, together with all cells weakly below it and in the same column, and its hooklength $h_\lambda(i,j)$ is the number of cells in its hook.
On the right in Figure~\ref{fig:HooklengthEx}, each cell of $\la=(4,4,2)$  is labeled with its hooklength. The hooklength formula for a diagram with $n$ cells is
\begin{equation}
\label{hook}
 f^\lambda = \dfrac{n!}{\prod h_{\lambda}(i,j)}, 
\end{equation}
where the product is over all cells $(i,j)$ in $\lambda$. Returning to our example shape $(4,4,2)$, we see that the corresponding number of SYT is
\[ f^{(4,4,2)} = \dfrac{10!}{6\cdot 5^2\cdot 4\cdot 3\cdot 2^3 \cdot 1^2} = 252. \]

\begin{figure}[t!]
\centering
$Y=$
\begin{ytableau}
1&3&4&8\\
2&6&7&9\\
5&10
\end{ytableau}
\hspace{50pt}
hooklengths:
\begin{ytableau}
6&5&3&2\\
5&4&2&1\\
2&1
\end{ytableau}
	\caption{A SYT of shape $(4,4,2)$ and a diagram of its hooklengths}
	\label{fig:HooklengthEx}
\end{figure}

\begin{lem}\label{lem:SYT_grid}
Let $m\geq 1$ and $n\geq 3$.
We can relate the number of standard Young tableaux of shape  $(n^m)$ and of shape $(n^{m-1},n-2)$ by the following equality:
\[ f^{(n^{m-1}, n-2)} = \frac{(n-1)(m+1)}{2(mn-1)} f^{(n^m)}. \]
\end{lem}

\begin{proof}
Let $\lambda=(n^m)$ and $\mu=(n^{m-1},n-2)$.  We will proceed by first describing which factors differ between $f^\lambda$ and $f^\mu$. We can observe that the hooklengths only disagree between $\lambda$ and $\mu$ in those cells in the last two columns and those in the last row. The last two columns of $\lambda$ have hooklengths of $m+1, m, \ldots, 2$ and $m, m-1, \ldots, 1$, while in $\mu$ the last two columns have hooklengths $m, m-1, \ldots, 2$ and $m-1, m-2, \ldots, 1$. Overall, $f^\mu$ is missing a factor of $(m+1)m$ which appears in the denominator of $f^\lambda$. Similarly, the hooklength values of the last row of $\lambda$, excluding the ones in the last two columns which have already been accounted for, are $n,n-1,\dots,3$, while those in $\mu$ are $n-2,n-3,\dots,1$.  So our formula for $f^\mu$ is missing a factor of $n(n-1)$ from the denominator and a factor of $2$ from the numerator.  Finally $f^\lambda$ has a numerator of $(mn)!$ while $\mu$ has a numerator of $(mn-2)!$, so there is a factor of $(mn)(mn-1)$ we need to remove from the numerator of $f^\lambda$.   Overall, our hooklength formula for $\mu$ derived from $f^\lambda$ is
\[ f^\mu \quad = \quad \dfrac{n(n-1)(m+1)m}{2(mn)(mn-1)} f^\lambda \quad = \quad \dfrac{(n-1)(m+1)}{2(mn-1)} f^\lambda, \]
as desired.
\end{proof}

\begin{thm}\label{thm:2chains}
Let $C_m$ and $C_n$ be chains of lengths $m\geq 2$ and $n\geq 2$, respectively. Then their product $C_m\times C_n$ has a $1/3$-balanced pair. 
\end{thm}

\begin{proof}
Without loss of generality, we can let $m\le n$. Let $P=C_m\times C_n$. If $m=2$, then $P$ has width 2, and so is $1/3$-balanced by a result of Linial (see Theorem~\ref{thm:width2}). If $m=n=3$, then $P$ has a non-trivial automorphism, and so by Theorem~\ref{thm:automorph}, $P$ has a $1/3$-balanced pair.

Next, let $m\geq 3$ and $n\geq 4$. Consider the cells $a=(1,2)$ and  $b=(2,1)$.  We claim that $(a,b)$ are a $1/3$-balanced pair. As discussed at the beginning of this subsection, $e(P)=f^\la$ and $e(P+ab)=f^\mu$, where $\la=(n^m)$ and $\mu=(n^{m-1},n-2)$. Hence, by Lemma~\ref{lem:SYT_grid}, we know that 
\[ e(P + ab) =  \frac{(n-1)(m+1)}{2(mn-1)}e(P). \]
It remains to be shown that
\begin{equation}
\frac{1}{3}\leq  \frac{(n-1)(m+1)}{2(mn-1)} \leq \frac{2}{3} \label{NMbdd}
\end{equation}
for all $m\geq 3$, $n\geq 4$. 
For the first inequality, cross multiply and bring everything to one side to get the equivalent inequality $(mn-1)+3(n-m)\ge0$. This inequality is true since $n\ge m$ and $mn\ge 1$.  

For the second inequality, proceed in the same manner to get $mn+3(m-n)-1\ge0$.  By the lower bounds for $m,n$ we have $(m-3)(n-4)\ge0$.  So it suffices to prove $mn+3m-3n-1\ge (m-3)(n-4)$.  Moving everything to one side yet again gives the equivalent inequality $7m-13\ge0$ which is true since $m\ge3$.

Therefore, we have shown that (\ref{NMbdd}) holds, and so $(a,b)$ is a $1/3$-balanced pair in $P$.
\end{proof}

\begin{question}
We were motivated in part to study products of chains as they are isomorphic to divisor lattices. Can one show that a product of $k$ chains is $1/3$-balanced, for $k\ge 3$?
\end{question}


\section{Other Diagrams}\label{sec:OtherDiagrams}

\begin{figure}[t!]
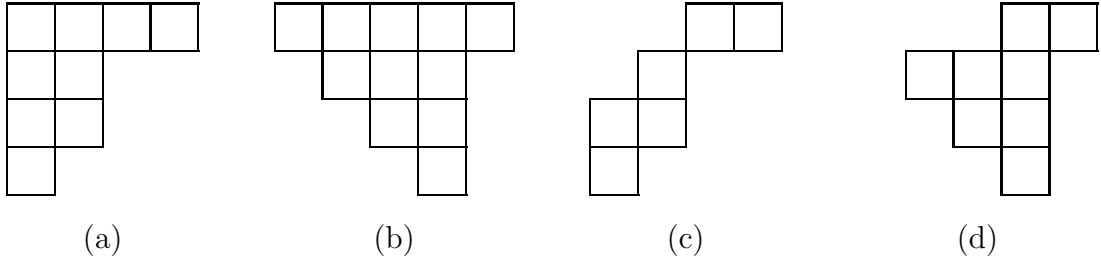

\centering
	\begin{minipage}{.2\textwidth}
		\centering
		\ydiagram{4,2,2,1}
		\begin{center}
		(a)
		\end{center}
	\end{minipage}
	\hspace{3mm}
	\begin{minipage}{.2\textwidth}
		\centering
		\ydiagram{5,1+3,2+2,3+1}
		\begin{center}
		(b)
		\end{center}
	\end{minipage}
	\hspace{3mm}
	\begin{minipage}{.2\textwidth}
		\centering
		\ydiagram{2+2,1+1,2,1}
		\begin{center}
		(c)
		\end{center}
	\end{minipage}
	\hspace{3mm}
	\begin{minipage}{.2\textwidth}
		\centering
		\ydiagram{3+2,1+3,2+2,3+1}
		\begin{center}
		(d)
		\end{center}
	\end{minipage}
	\caption{(a) A left-justified Young diagram of shape $(4,2^2,1)$, (b) a shifted diagram of shape $(5,3,2,1)$, (c) a skew left-justified diagram of shape $(4,2^2,1)\ /\ (2,1)$, \\and (d) a shifted skew diagram of shape $(5,3,2,1)\ /\ (3)$}
	\label{fig:ManyDiagrams}
\end{figure}

In Section~\ref{sec:2Chains}, we considered the product of two chains as a rectangular Young diagram, and the linear extensions of the poset corresponded to the standard Young tableaux of that Young diagram.  
Given Theorem~\ref{thm:2chains}, it is natural  to consider other posets that come from other diagrams. 
Suppose  $\lambda_1 > \lambda_2> \cdots > \lambda_k$, in which case $\lambda=(\la_1,\dots,\la_k)$ is called a 
{\em strict partition}. 
The \emph{shifted diagram} corresponding to a strict partition $\lambda$ indents row $i$ so that it begins at the diagonal cell $(i,i)$. An example is given in Figure~\ref{fig:ManyDiagrams}(b). A third type of diagram is a \emph{skew diagram}, $\lambda/\mu$, which is the set-theoretic difference between diagrams $\lambda=(\lambda_1, \ldots, \lambda_k)$ and $\mu=(\mu_1,\ldots,\mu_l)$ such that $\mu\subseteq \lambda$, that is, $l\leq k$ and $\mu_i\leq \lambda_i$ for each $1\leq i\leq l$. A skew diagram can be either \emph{left-justified}, as seen in Figure~\ref{fig:ManyDiagrams}(c), or \emph{shifted}, as seen in Figure~\ref{fig:ManyDiagrams}(d).
Note that when $\mu$ is empty then $\la/\mu=\la$.  Also, we will now use the term ``Young diagram" to refer to any of the four possibilities we have described.

Any Young diagram can be turned into a poset using the same ordering on the cells as before.  In addition, a {\em standard Young tableau} can be obtained from a diagram with $n$ boxes by filling them with the numbers $1,\dots, n$ so that rows and columns increase.  Such tableaux correspond bijectively to linear extensions of the corresponding poset.
We next present a generalized version of Theorem~\ref{thm:2chains} for arbitrary shapes.

\begin{thm}
Let $P_{\lambda/\mu}$ be the poset corresponding to the Young diagram $\lambda/\mu$.  If $P_{\lambda/\mu}$ is not a chain, then it is $1/3$-balanced.
\end{thm}

\begin{proof}
Let $\lambda = (\lambda_1, \ldots, \lambda_k)$ and $\mu=(\mu_1,\ldots, \mu_l)$.
Assume first that $\mu$ is empty. We will show that $P_\lambda$  has an almost twin pair of elements and so, by Theorem~\ref{thm:asymmetricpair},  is $1/3$-balanced.

When $\lambda$ is left-justified, let $x$ correspond to the $(1,2)$ cell and $y$ correspond to the $(2,1)$ cell of $\lambda$. 
Both of these cells must exist in $\la$ since $P_\la$ is not a chain.  It is now easy to verify that $(x,y)$ is an almost twin pair.
If $\lambda$ is a shifted, then $\lambda_1\geq 3$ and $\la_2\ge1$, as $P_\lambda$ is not a chain. Let $x$ correspond to the $(1,3)$ cell and $y$ correspond to the $(2,2)$ cell. Again, $(x,y)$ is an almost twin pair of elements in $P_\lambda$.

\begin{figure}[t!]
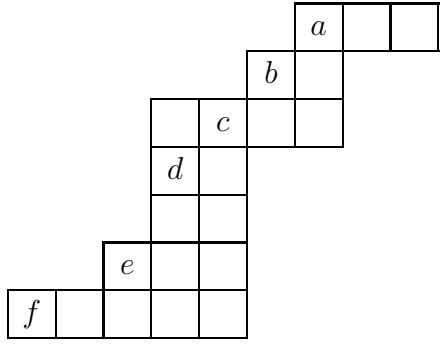

\centering
	\begin{minipage}{\textwidth}
		\centering
		\ytableaushort{\none\none\none\none\none\none a,\none\none\none\none\none b,\none\none\none\none c,\none\none\none d,\none,\none\none e,f}
		*{6+3,5+2,3+4,3+2,3+2,2+3,5}
		\caption{The skew diagram $(9,7^2,5^4)\ /\ (6,5,3^3,2)$}
		\label{fig:SkewForProof}
	\end{minipage}
\end{figure}

Next, we consider skew diagrams. If $\lambda/\mu$ is a disconnected diagram, observe that an almost twin pair in a connected component of $P_{\lambda/\mu}$ remains an almost twin pair in the entire poset. Therefore, we can assume $\lambda/\mu$ is a connected skew  diagram that does not correspond to a poset that is a chain. First consider  skew left-justified diagrams. By removing any empty columns on the left of the diagram, we can assume without loss of generality that $k\geq l+1$. For ease in discussing the first and last rows of $\mu$, define $\mu_0 = \lambda_1$ and $\mu_{l+1} = 0$. We have the following cases:

\begin{enumerate}[(i)]
\item If there exists $i \in [l]$ such that 
\[ \mu_{i-1}-1 \geq \mu_i = \mu_{i+1}+1 \]
then $(i,\mu_i+1)$ and $(i+1,\mu_{i+1}+1)$ is an almost twin pair.   For an example of this case, see the pair $(a,b)$ in Figure~\ref{fig:SkewForProof}.

\item If there exists $i\in [l-1]$ such that
\[ \mu_{i-1}-2 \geq \mu_i = \mu_{i+1}  \]
then $(i, \mu_{i}+2)$ and $(i+1,\mu_{i+1}+1)$ is an almost twin pair. Note that $(i,\mu_i+2)$ exists in the diagram since $\lambda/\mu$ is connected. For an example of this, see the pair $(c,d)$ in Figure~\ref{fig:SkewForProof}.

\item If $k=l+1$ and $\mu_{l-1} -1 \geq \mu_l$, then $(l,\mu_l+1)$ and $(l+1,1)$ are an almost twin pair. For an example of this, see the pair $(e,f)$ in Figure~\ref{fig:SkewForProof}.

\item If $k\geq l+2$ and $\mu_l\geq 2$, then $(l+1,2)$ and $(l+2,1)$ are an almost twin pair. Notice this is similar to case (ii), only it occurs at the bottom of the skew diagram.
\end{enumerate}

We can now decide what types of diagrams do not fall into cases (i)-(iv) above. We claim that any remaining diagram has $\mu$ of the form 
\[(s^{m_1}, (s-1)^{m_2}, \ldots, (s-p+1)^{m_p})\]
where $1\le p\le s$ and $m_i\geq 2$ for all $i\in [p]$. We call this case (v). 

\begin{figure}[t!]
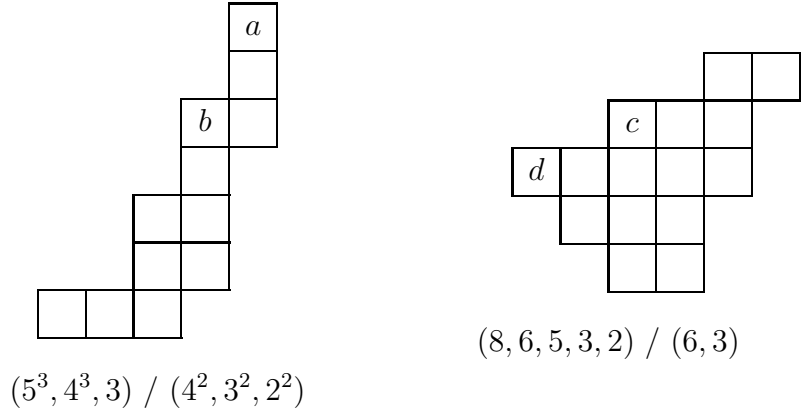

\centering
	\begin{minipage}{.3\textwidth}
		\centering
		\ytableaushort{\none\none\none\none a,\none, \none\none\none b\none, \none, \none, \none, \none\none\none}
		*{4+1,4+1,3+2,3+1,2+2,2+2,3}
		\[ (5^3,4^3,3)\ /\ (4^2,3^2,2^2) \]		
	\end{minipage} 
\hspace{6mm}
	\begin{minipage}{.32\textwidth}
		\centering
		\ytableaushort{\none\none\none\none\none\none, \none\none\none\none c, \none\none d}
			*{6+2,4+3,2+5,3+3,4+2}
		\[ (8,6,5,3,2)\ /\ (6,3) \]
	\end{minipage}
	\caption{A skew left-justified diagram and a skew shifted diagram}
	\label{fig:SkewCase}
\end{figure}

To verify the claim, note that all consecutive $\mu_i$ values differ by $1$ or $0$ since if there is some $r$ with $\mu_{r-1} -2\geq \mu_r$, then to avoid cases (i) and (ii) above, it must be that $\mu_{s-1}-2\geq \mu_s$ for all $s\in [r,l+1]$. In particular, this means $\mu_l \ge 2$, and this diagram will fall into case (iii) or (iv). So, any consecutive $\mu_i$ values differ by $1$ or $0$. Further, if $m_i = 1$ for any $i\in [p]$, the diagram would fall into case (i). Hence, $\mu$ must have the form above.

It also must be true that $\lambda/\mu$ has $\lambda_1 = \mu_1+1$, in order to avoid case (ii) above. An example  of a diagram $\lambda/\mu$ that does not fall into cases (i)-(iv) is given in Figure~\ref{fig:SkewCase}. 
In these remaining diagrams, $(1,\mu_1+1)$ and $(m_1+1, \mu_{(m_1+1)} +1)$ is an almost twin pair. An example of such a pair is $(a,b)$ in Figure~\ref{fig:SkewCase}. Hence every skew left-justified diagram $\lambda/\mu$ satisfies one of these five cases, and so $P_{\lambda/\mu}$ has an almost twin pair.

Finally, we consider the skew shifted diagrams. Notice that that the first $l$ rows of the diagram can be viewed as a skew left-justified diagram. Therefore, if any of the first $l-1$ rows are of the forms found in cases (i) or (ii), or if the first rows correspond to case (v), then the almost twin pairs in those cases remain almost twin in this poset, and we are done.

If none of cases (i), (ii), or (v) apply, then consider $\mu_l$. In particular, it must be the case that $\mu_l >1$, else case (ii) or (v) applies. If $\mu_l >3$, then the last $k-l$ rows of the diagram are a shifted diagram, and so we have the same almost twin pair as in the shifted case. If $\mu_l \in \{2,3\}$, then $(l, \mu_l+l)$ and $(l+1, l+1)$ are an almost twin pair, as seen by the pair $(c,d)$ in Figure~\ref{fig:SkewCase}. Hence, for skew diagrams $\lambda/\mu$, if $P_{\lambda/\mu}$ is not a chain, it has an almost twin pair of elements, as claimed. Hence, even when $\mu$ is not empty, we can always find an almost twin pair.
\end{proof}


\section{Posets of Dimension $2$}\label{sec:Dim2}

The set of linear extensions $E(P)$ of a labeled poset $P$ with $n$ elements can be considered as a subset of the symmetric group $\fS_n$, where permutations are written in one-line notation. 
The \emph{dimension} of a poset is the least $k$ such that there is some $U\subseteq E(P)$ of size $k$ such that $\cap U = (P,\le)$. An equivalent definition is that the dimension of $P$ is the least $k$ such that $P$ can be embedded into the product $\bbN^k$ where $\bbN=\{0,1,2,\dots\}$.

We will be concentrating on posets of dimension $2$. Every such poset can be realized using a permutation $\pi=\pi_1\dots\pi_n$ since a poset of dimension $2$ can be obtained by intersecting the linear order in $\pi$ with the natural order $1<\dots<n$ on the integers. We will use $P_\pi$ to denote this poset.  Figure~\ref{fig:Dim2_example} displays the poset $P_\pi$ when $\pi=41325$.

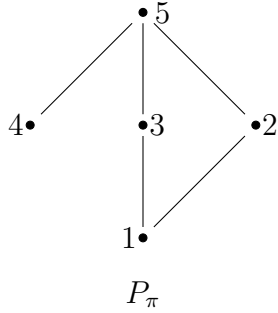
\begin{figure}
\centering
	\begin{minipage}{.35\textwidth}
		\centering
		\begin{tikzpicture}[scale=0.75]
		\filldraw (0,0) circle (2pt) node(v1){};
		\filldraw (2,2) circle (2pt) node(v2){};
		\filldraw (0,2) circle (2pt) node(v3){};
		\filldraw (-2,2) circle (2pt) node(v4){};
		\filldraw (0,4) circle (2pt) node(v5){};
		\draw (-0.25,0) node {$1$};
		\draw (2.25,2) node {$2$};
		\draw (0.25,2) node {$3$};
		\draw (-2.25,2) node {$4$};
		\draw (0.35,4) node {$5$};
		\node (x) at (0,-1) {$P_\pi$};
		\draw[-] (v5)--(v2)--(v1)--(v3)--(v5)--(v4);
		\end{tikzpicture}
	\end{minipage}
	\caption{The poset $P_\pi$ for $\pi=41325$}
	\label{fig:Dim2_example}
\end{figure}

To state our result, we will need some definitions from the theory of permutation patterns.  If $\pi$ and $\sigma$ are permutations then we say that $\pi$ {\em contains a copy of $\sigma$} if there is some subsequence of $\pi$ whose elements are in the same relative order as those of $\sigma$.  Otherwise, $\pi$ {\em avoids} $\sigma$.   For example, if $\pi=23154$ then $\pi$ contains the pattern $132$ because the subsequence $254$, like the pattern, has its smallest element first, its largest element second, and its middle-sized element last.  On the other hand, $\pi$ avoids $321$ since it does not contain a decreasing subsequence with three elements.  Given a subsequence $\pi'$ of $\pi$ we say that $\pi'$ is {\em contained in a copy of $\sigma$} if some copy of $\sigma$ in $\pi$ uses every element  of $\pi'$ (and perhaps others).  Otherwise, we say $\pi'$ {\em avoids} $\sigma$.  Note that $\pi'$ can avoid $\sigma$ even if $\pi$ contains it.  Returning to our example, $\pi'=14$ is contained in the pattern $132$ because of the subsequence $154$ of $\pi$.  But $\pi'$ avoids $123$ since none of the copies of $123$ in $\pi$ use the $1$.

Finally, define an {\em inversion} in $\pi=\pi_1\dots\pi_n$ to be a copy $\pi_i\pi_j$ of the pattern $21$.  Note that in this case some authors define the inversion to be the pair of corresponding indices $(i,j)$.

\begin{prop}\label{prop:perm_312_231}
Let $\pi=\pi_1\pi_2\ldots \pi_n$ be an element of $\mathfrak{S}_n$, and assume that $\pi$ has an inversion  $\pi_i\pi_j$ avoiding the patterns $312$ and $231$ in $\pi$. Then the pair $(\pi_i,\pi_j)$ is $1/2$-balanced in $P_\pi$.
\end{prop}

Before we proceed to the proof, we can observe an example of this in Figure~\ref{fig:Dim2_example}. Note that $32$  is an inversion of $\pi = 41325$ which avoids $312$ and $231$ in $\pi$. So $(3,2)$ is a $1/2$-balanced pair in $P_\pi$. 

\begin{proof}
To simplify notation, let $\pi_i=y$ and $\pi_j=x$. Therefore, $\pi$ has the form 
\[ \pi = \pi_1 \cdots\ y\ \cdots\ x\ \cdots \pi_n \]
where $y>_\bbN x$. Now, since $yx$ avoids $312$ and $231$, there are no elements between $x$ and $y$ in $\pi$ that are larger than $y$ or smaller than $x$. Also, no elements to the right of $x$ or left of $y$ have values between $x$ and $y$. To put this description another way, if $yx$ avoids $312$ and $231$ in $\pi$, the elements between $y$ and $x$ in $\pi$ are exactly those in the set $\{a\mid x<_\bbN a<_\bbN y\}$.

We claim that $U_x=U_y$ and $L_x=L_y$ in $P_\pi$. We will show that $U_x = U_y$ as the proof of $L_x = L_y$ is nearly identical. If $z\in U_x$, then $z$ is to the right of $x$ in $\pi$ and thus also to the right of $y$ in $\pi$. Since $yx$ avoids $312$ in $\pi$ and $x<_\bbN z$, it must be that $y<_\bbN z$. Hence, $y<_P z$ and so $z \in U_y$.

If $z\in U_y$, then $z$ is to the right of $y$ in $\pi$ and $y<_\bbN z$. Since $yx$ avoids $231$ in $\pi$, then $z$ must also be to the right of $x$ in $\pi$. Also, $x<_\bbN y<_\bbN z$. Thus $x<_P z$, which means $z\in U_x$. Hence, we have that $U_x = U_y$.

Now, because $U_x = U_y$ and $L_x=L_y$, $(x,y)$ is a twin pair of elements. So, by Corollary~\ref{cor:twinpair}, $P_\pi$ is $1/2$-balanced as desired.
\end{proof}

Not every poset $P_\pi$ is $1/2$-balanced.  For example, if $\pi=  13572468$ then $P_\pi$ is isomorphic to the distributive lattice $J(P)$ in Fig.~\ref{fig:J234Chain}.  But  we have already noted that this lattice is not $1/2$-balanced.   
\begin{question}
Do posets of dimension $2$ satisfy Conjecture~\ref{conj}?
\end{question}

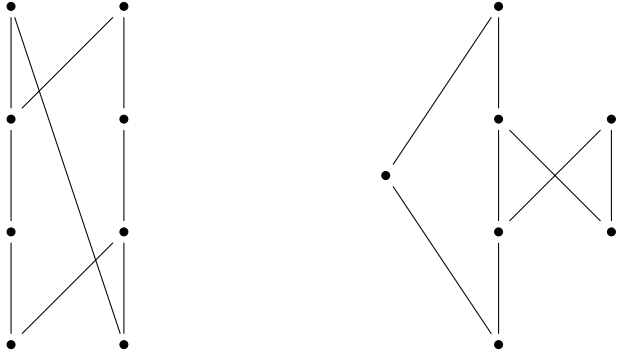
\begin{figure}
\centering
	\begin{minipage}{.3\textwidth}
	\centering
		\begin{tikzpicture}[scale=0.75]
		\filldraw (0,0) circle (2pt) node(v2){};
		\filldraw (2,0) circle (2pt) node(v1){};
		\filldraw (2,2) circle (2pt) node(v3){};
		\filldraw (2,4) circle (2pt) node(v5){};
		\filldraw (2,6) circle (2pt) node(v7){};
		\filldraw (0,2) circle (2pt) node(v4){};
		\filldraw (0,4) circle (2pt) node(v6){};
		\filldraw (0,6) circle (2pt) node(v8){};
		\draw[-] (v1)--(v3)--(v5)--(v7)--(v6)--(v4)--(v2)--(v3);
		\draw[-] (v1)--(v8)--(v6);
		\end{tikzpicture} 
	\end{minipage}
	\hspace{5mm}
	\begin{minipage}{.3\textwidth}
	\centering
	\begin{tikzpicture}[scale=0.75]
		\filldraw (0,0) circle (2pt) node(v1){};
		\filldraw (2,2) circle (2pt) node(v3){};
		\filldraw (2,4) circle (2pt) node(v7){};
		\filldraw (-2,3) circle (2pt) node(v5){};
		\filldraw (0,2) circle (2pt) node(v2){};
		\filldraw (0,4) circle (2pt) node(v4){};
		\filldraw (0,6) circle (2pt) node(v6){};
		\draw[-] (v1)--(v2)--(v7)--(v3)--(v4)--(v2);
		\draw[-] (v1)--(v5)--(v6)--(v4);
		\end{tikzpicture} 
	\end{minipage}
	\caption{Two posets with small balance constants}
	\label{fig:SmallDeltas}
\end{figure}


\section{Posets with Small Balance Constants}\label{sec:SmallDeltas}

It would be of interest to characterize those posets whose balance constant is exactly $1/3$, or to see if there are posets satisfying Conjecture~\ref{conj} with $\delta(P)$ arbitrarily close to $1/3$.  For the second question, people have considered posets of width $2$ (where width is the largest cardinality of an antichain) because  Linial~\cite{Lin84} proved that these posets satisfy the conjecture.

\begin{thm}[\cite{Lin84}]~\label{thm:width2}
Let $(P,\leq)$ be a poset of width $2$. Then, $\delta(P) \geq 1/3$.\hfill \qed
\end{thm}

Aigner~\cite{Aig85} showed that posets of width $2$ fit into one of two categories: either the poset is a linear sum of copies of the singleton poset and $T$ (the poset from Figure~\ref{fig:3elem1relation}); or the poset has an $\alpha$-balanced pair with $1/3< \alpha< 2/3$. In fact, the only known posets that have a balance constant of $1/3$ are the linear sums of singletons and $T$. The poset of width $2$ in Figure~\ref{fig:SmallDeltas} has a balance constant of $\frac{16}{45}\approx 0.3556$, and until recently, it was the poset with the smallest known balance constant greater than $1/3$~\cite{Bri99}.  Using computer search, we have found posets of width $2$ that have balance constants closer to $1/3$. In Figure~\ref{fig:SmallestDeltas}, the three posets have balance constants $\delta(A) = \frac{6}{17}\approx 0.35294$,  $\delta(B) = \frac{60}{171} \approx 0.350877$, and  $\delta(C) = \frac{37}{106} \approx 0.349057$.

The smallest known balance constant for a poset with width strictly greater than 2 is $\frac{14}{39} \approx 0.3590$, as described in~\cite{Bri99}. It belongs to the poset with $7$ elements in Figure~\ref{fig:SmallDeltas}. A computer search of all posets with up to $9$ elements revealed no posets with balance constant smaller that $\frac{14}{39}$ and width greater than 2. These observations raise the following questions.

\begin{figure}[t!]
\centering
	\begin{minipage}{.3\textwidth}
	\centering
		\begin{tikzpicture}[scale=0.75]
		\filldraw (0,0) circle (3pt) node(v1){};
		\filldraw (0,2) circle (3pt) node(v2){};
		\filldraw (0,4) circle (3pt) node(v7){};
		\filldraw (0,6) circle (3pt) node(v8){};
		\filldraw (0,8) circle (3pt) node(v9){};
		\filldraw (2,1) circle (3pt) node(v3){};
		\filldraw (2,3) circle (3pt) node(v4){};
		\filldraw (2,5) circle (3pt) node(v5){};
		\filldraw (2,7) circle (3pt) node(v6){};
		\node (va) at (1,-1) {$A$};
		\draw[-] (v4)--(v1)--(v2)--(v7)--(v8)--(v9)--(v5);
		\draw[-] (v3)--(v4)--(v5)--(v6);
		\end{tikzpicture} 
	\end{minipage}
	\hspace{5mm}
	\begin{minipage}{.3\textwidth}
	\centering
	\begin{tikzpicture}[scale=0.75]
		\filldraw (0,0) circle (3pt) node(v1){};
		\filldraw (0,2) circle (3pt) node(v2){};
		\filldraw (0,4) circle (3pt) node(v4){};
		\filldraw (0,6) circle (3pt) node(v7){};
		\filldraw (0,8) circle (3pt) node(v9){};
		\filldraw (0,10) circle (3pt) node(v11){};
		\filldraw (2,1) circle (3pt) node(v3){};
		\filldraw (2,3) circle (3pt) node(v5){};
		\filldraw (2,5) circle (3pt) node(v6){};
		\filldraw (2,7) circle (3pt) node(v8){};
		\filldraw (2,9) circle (3pt) node(v10){};
		\node (vb) at (1,-0.5) {$B$};
		\draw[-] (v3)--(v5)--(v7)--(v9)--(v6)--(v5)--(v1)--(v2)--(v6);
		\draw[-] (v2)--(v4)--(v8)--(v10);
		\draw[-] (v6)--(v8)--(v11)--(v9);
		\draw[-] (v4)--(v7);
		\end{tikzpicture} 
	\end{minipage}
	\hspace{5mm}
	\begin{minipage}{.3\textwidth}
	\centering
	\begin{tikzpicture}[scale=0.75]
		\filldraw (0,0) circle (3pt) node(v2){};
		\filldraw (0,2) circle (3pt) node(v4){};
		\filldraw (0,4) circle (3pt) node(v6){};
		\filldraw (0,6) circle (3pt) node(v8){};
		\filldraw (0,8) circle (3pt) node(v10){};
		\filldraw (2,0) circle (3pt) node(v1){};
		\filldraw (2,2) circle (3pt) node(v3){};
		\filldraw (2,4) circle (3pt) node(v5){};
		\filldraw (2,6) circle (3pt) node(v7){};
		\filldraw (2,8) circle (3pt) node(v9){};
		\node (vc) at (1,-1) {$C$};
		\draw[-] (v1)--(v3)--(v5)--(v7)--(v9)--(v8)--(v3)--(v2);
		\draw[-] (v2)--(v4)--(v5);
		\draw[-] (v4)--(v6)--(v8)--(v10);
		\draw[-] (v6)--(v7);
		\end{tikzpicture} 
	\end{minipage}
	\caption{Posets with the smallest balance constants greater than $1/3$}
	\label{fig:SmallestDeltas}
\end{figure}
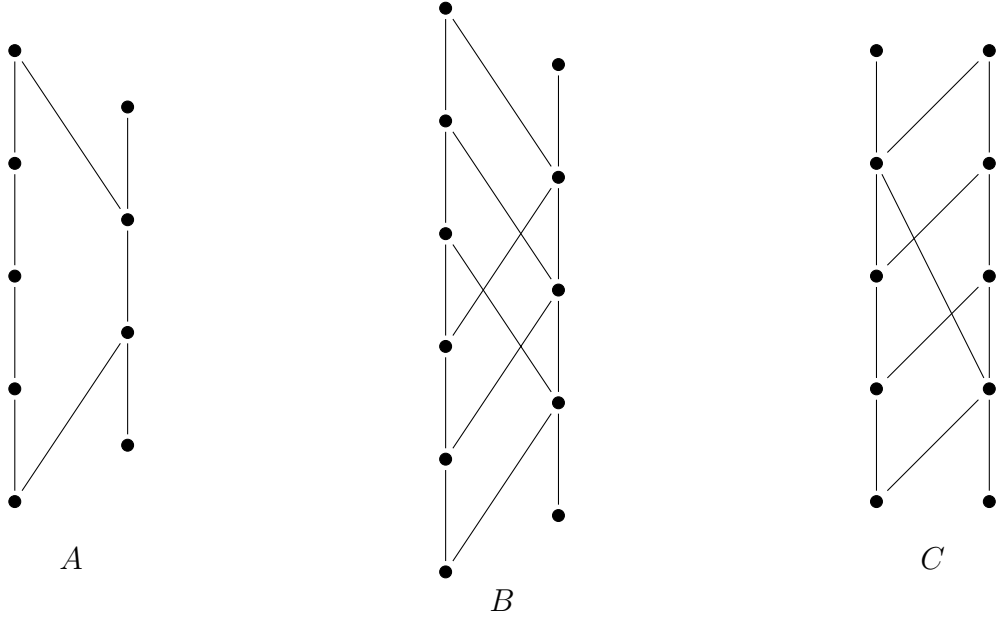

\begin{question}
Can one find a sequence of posets whose balance constants approach $1/3$?  If $P$ has width $w$, is there always a poset $Q$ of smaller width such that $\delta(Q)<\delta(P)$?
\end{question}

%
%
%



\nocite{*}
\bibliographystyle{alpha}

\bibliography{thesisbib}

\begin{thebibliography}{{Zag}16}

\bibitem[Aig85]{Aig85}
Martin Aigner.
\newblock A note on merging.
\newblock {\em Order}, 2(3):257--264, 1985.

\bibitem[BFT95]{BFT95}
Graham Brightwell, S.~Felsner, and W.~T. Trotter.
\newblock Balancing pairs and the cross product conjecture.
\newblock {\em Order}, 12(4):327--349, 1995.

\bibitem[Bri89]{Bri89}
Graham Brightwell.
\newblock Semiorders and the 1/3--2/3 conjecture.
\newblock {\em Order}, 5(4):369--380, 1989.

\bibitem[Bri99]{Bri99}
Graham Brightwell.
\newblock Balanced pairs in partial orders.
\newblock {\em Discrete Mathematics}, 201(1-3):25--52, 1999.

\bibitem[Fre76]{Fre76}
Michael~L. Fredman.
\newblock How good is the information theory bound in sorting?
\newblock {\em Theoretical Computer Science}, 1(4):355 -- 361, 1976.

\bibitem[GHP87]{GHP}
B.~Ganter, G.~Hafner, and W.~Poguntke.
\newblock On linear extensions of ordered sets with a symmetry.
\newblock {\em Discrete Mathematics}, 63:153--156, 1987.

\bibitem[Kis68]{Kis68}
S.S Kislitsyn.
\newblock Finite partially ordered sets and their associated sets of
  permutation.
\newblock {\em Matematicheskiye Zametki}, 4:511--518, 1968.

\bibitem[KS84]{KS84}
Jeff Kahn and Michael Saks.
\newblock Balancing poset extensions.
\newblock {\em Order}, 1(2):113--126, 1984.

\bibitem[Lin84]{Lin84}
Nathan Linial.
\newblock The information-theoretic bound is good for merging.
\newblock {\em SIAM Journal on Computing}, 13(4):795 -- 801, 1984.

\bibitem[Pec06]{Pec06}
Marcin Peczarski.
\newblock The gold partition conjecture.
\newblock {\em Order}, 23(1):89--95, 2006.

\bibitem[Pec08]{Pec08}
Marcin Peczarski.
\newblock The gold partition conjecture for 6-thin posets.
\newblock {\em Order}, 25(2):91--103, 2008.

\bibitem[TGF92]{TGF92}
W.~T. Trotter, W.~G. Gehrlein, and P.~C. Fishburn.
\newblock Balance theorems for height-2 posets.
\newblock {\em Order}, 9(1):43--53, 1992.

\bibitem[Zag12]{Zag12}
Imed Zaguia.
\newblock The $1/3-2/3$ conjecture for $n$-free ordered sets.
\newblock {\em Elec. J. Combinatorics}, 19(2):1--5, 2012.

\bibitem[{Zag}16]{Zag16}
I.~{Zaguia}.
\newblock {The 1/3-2/3 Conjecture for ordered sets whose cover graph is a
  forest}.
\newblock {\em ArXiv}, 1610.00809, 2016.

\end{thebibliography}

\end{document}